\documentclass{article}
\usepackage{amsmath,amssymb, amsthm}
\newtheorem{theorem}{Theorem}[section]
\newtheorem{proposition}[theorem]{Proposition}
\newtheorem{definition}[theorem]{Definition}
\newtheorem{example}[theorem]{Example}
\newtheorem{lemma}[theorem]{Lemma}
\newtheorem{corollary}[theorem]{Corollary}

 \def\ivq{\in\!\vee\,q}
 \def\iwq{\in \!\wedge\,q}

\begin{document}
\title{\bf Interval valued  $(\in,\ivq)$-fuzzy filters of pseudo $BL$-algebras }
 \normalsize

\author{{Jianming Zhan}$^{a,*}$, {Wies{\l}aw A. Dudek}$^b$, {Young Bae Jun}$^c$\\
 {\small $^a$ Department of Mathematics, Hubei Institute for Nationalities,}\\
 {\small Enshi, Hubei Province, 445000, P. R. China}\\
  {\small $^b$ Institute of Mathematics and Computer Science,
  Wroc{\l}aw University of Technology,}\\
  {\small Wybrze\.ze Wyspia\'nskiego 27, 50-370 Wroc{\l}aw, Poland}  \\
  {\small  $^c$  Department of Mathematics Education,
   Gyeongsang National University,}\\
  {\small Chinju 660-701, Korea }
  }

\date{}\maketitle

\begin{flushleft}\rule[0.4cm]{12cm}{0.3pt}
\parbox[b]{12cm}{\small
\bf Abstract\rm \paragraph{ } We introduce the concept of
quasi-coincidence of a fuzzy interval value with an interval
valued fuzzy set. By using this new idea, we introduce the notions
of interval valued $(\in,\ivq)$-fuzzy filters of pseudo
$BL$-algebras and investigate some of their related properties.
Some characterization theorems of these generalized interval
valued fuzzy  filters are derived. The relationship among these
generalized interval valued fuzzy filters of pseudo $BL$-algebras
is considered. Finally, we consider the concept of
implication-based interval valued fuzzy implicative filters of
pseudo $BL$-algebras, in particular, the implication
operators in Lukasiewicz system of continuous-valued logic are discussed.  \\


} \rule{12cm}{0.3pt}
\end{flushleft}
\footnote{* Corresponding author.\\
 \it E-mail address:\rm \ \
  zhanjianming@hotmail.com (J. Zhan), dudek@im.pwr.wroc.pl (W. A. Dudek)}

\section{Introduction}
Logic appears in a ''scared'' form (resp., a ''profane'') which is
dominant in proof theory (resp., model theory). The role of logic
in mathematics and computer science is twofold-as a tool for
applications in both areas, and a technique for laying the
foundations. Non-classical logic including many-valued logic,
fuzzy logic, etc., takes the  advantage of the classical logic to
handle information with various facets of uncertainty (see
\cite{29} for generalized theory of uncertainty), such as
fuzziness, randomness, and so on.  Non-classical logic has become
a formal and useful tool for computer science to deal with fuzzy
information and uncertain information. Among all kinds of
uncertainties, incomparability is an important one which can be
encountered in our life.

$BL$-algebras were introduced by H\'ajek as algebraic structures
for his Basic Logic, starting from continuous $t$-norm and their
residuals (\cite{18}). $MV$-algebras \cite{4}, product algebras
and G\"odel algebras are the most classes of $BL$-algebras.
Filters theory play an important role in studying these algebras.
From logical point of view, various filters correspond to various
sets of provable formulae. H\'ajek \cite{18} introduced the
concepts of (prime) filters of $BL$-algebras. Using prime filters
of $BL$-algebras, he proved the completeness of Basic Logic $BL$.
$BL$-algebras are further discussed by Di Nola (\cite{9} and
\cite{10}), Iorgulescu (\cite{19}), Ma(\cite{20}) and Turunen
(\cite{24}, \cite{25}), and so on.

Recent investigations are concerned with non-commutative
generalizations for these structures (see [8, 11 - 17, 22 - 25, 32
- 33]). In \cite{16}, Georgescu et al. introduced the concept of
pseudo $MV$-algebras as a non-commutative generalization of
$MV$-algebras. Several researchers discussed the properties of
pseudo $MV$-algebras( see \cite{11}, \cite{12}, \cite{22} and
\cite{23}). Pseudo $BL$-algebras are a common extension of
$BL$-algebras and pseudo $MV$-algebras (see \cite{8}, \cite{13},
\cite{14}, \cite{17}, \cite{31}). These structures seem to be a
very general algebraic concept in order to express the
non-commutative reasoning. We remark that a pseudo $BL$-algebra
has two implications and two negations.

After the introduction of fuzzy sets by Zadeh \cite{27}, there
have been a number of generalizations of this fundamental concept.
In \cite{28}, Zadeh made an extension of the concept of a fuzzy
set (i.e., a fuzzy set with an interval valued membership
function). The interval valued fuzzy subgroups were first defined
and studied by Biswas \cite{3} which are the subgroups of the same
natural of the fuzzy subgroups defined by Rosenfeld. A new type of
fuzzy subgroup, that is, the $(\in,\ivq)$-fuzzy subgroup, was
introduced in an earlier paper of Bhakat and Das \cite{2} by using
the combined notions of ''belongingness'' and ''quasicoincidence''
of fuzzy points and fuzzy sets, which was introduced by Pu and Liu
\cite{21}. In fact, the $(\in,\ivq)$-fuzzy subgroup is an
important generalization of Rosenfeld's fuzzy subgroup. Recently,
Davvaz \cite{5} applied this theory to near-rings and obtained
some useful results. Further, Davvaz and Corsini \cite{6}
redefined fuzzy $H_v$-submodule and many valued implications. In
\cite{33}, Zhan et al. also discussed the properties of interval
valued $(\in,\ivq)$-fuzzy hyperideals in hypernear-rings. For more
details, the reader is referred to \cite{5}, \cite{6} and
\cite{33}.

The paper is organized as follows. In section 2, we recall some
basic definitions and results of pseudo $BL$-algebras. In section 3,
we introduce the notion of interval valued $(\in,\ivq)$-fuzzy
filters in pseudo $BL$-algebras and investigate some of their
related properties. Further, the notions of interval valued
$(\in,\ivq)$-fuzzy (implicative, $MV$- and $G$-) of pseudo
$BL$-algebras are introduced and the relationship among these
generalized interval valued fuzzy filters of pseudo $BL$-algebras is
considered in section 4. Finally, in section 5 we consider the
concept of implication-based interval valued fuzzy implicative
filters of pseudo $BL$-algebras, in particular, the implication
operators in Lukasiewicz system of continuous-valued logic are
discussed.

\section{Preliminaries }

A pseudo $BL$-algebra is an algebra
$(A;\wedge,\vee,\odot,\rightarrow,\hookrightarrow,0,1)$ of type
(2,2,2,2,2, 0,0) such that $(A,\wedge,\vee,0,1)$ is a bounded
lattice, $(A,\odot,1)$ is a monoid and the following axioms

\begin{enumerate}
\item[$(a_1)$] \ $ x\odot y\le z\Longleftrightarrow x\le y\rightarrow
z\Longleftrightarrow y\le x \hookrightarrow z;$

\item[$(a_2)$] \ $x\wedge y=(x\rightarrow y)\rightarrow x=x\odot
(x\hookrightarrow y);$

\item[$(a_3)$] \  $(x\rightarrow y)\vee (y\rightarrow
x)=(x\hookrightarrow y)\vee (y\hookrightarrow x)=1$
\end{enumerate}

\noindent are satisfied for all $x,y,z\in A$.

We assume that the operations $\vee$, $\wedge$, $\odot$ have
priority towards the operations $\rightarrow$ and
$\hookrightarrow$.

\begin{example}\label{Ex2.1}\rm (Di Nola \cite{8}) Let
$(G,\vee,\wedge,+,-,0)$ be an arbitrary $l$-group and let $\theta$
be the symbol distinct from the element of $G$. If $G^-=\{x'\in
G\vert x'\le 0\}$, then we define on $G^*=\{\theta\}\cup G^-$  the
following operations:

$$x'\odot y'=\left\{\begin{array}{ll} x'+y' & \mbox{\ \ \ \ \
if\ \ \ \ \ } x',y'\in G^-,\\ \theta & \mbox{\ \ \ \ \
}$otherwsie$,
\end{array}\right.$$

$$x'\rightarrow y'=\left\{\begin{array}{ll} (y'-x')\wedge 0 & \mbox{\ \ \ \ \
if\ \ \ \ \ } x',y'\in G^-,\\ \theta & \mbox{\ \ \ \ \ if\ \ \ \ \
} x'\in G^-, \ y'=\theta,\\ 0 & \mbox{\ \ \ \ \ if\ \ \ \ \ } x'
=\theta,
\end{array}\right.$$

$$x'\hookrightarrow y'=\left\{\begin{array}{ll} (-x'+y')\wedge 0 & \mbox{\ \ \ \ \
if\ \ \ \ \ } x',y'\in G^-,\\ \theta & \mbox{\ \ \ \ \ if\ \ \ \ \
} x'\in G^-, \ y'=\theta,\\ 0 & \mbox{\ \ \ \ \ if\ \ \ \ \ } x'
=\theta.
\end{array}\right.$$

If we put $\theta\le x'$, for any $x'\in G$, then $(G^*,\le)$
becomes a lattice with first element $\theta$, and the last
element $0$. The structure
$G^*=(G^*,\vee,\wedge,\odot,\rightarrow,\hookrightarrow,0=\theta,1=0)$
is a pseudo $BL$-algebra. \end{example}

Let $A$ be a pseudo $BL$-algebra and $x,y,z\in A.$ The following
statements are true (for details see \cite{8, 15, 31}):

\begin{enumerate}
\item[$(1)$] \ $(x\odot y)\rightarrow z=x\rightarrow(y\rightarrow
z)$,

\item[$(2)$] \ $(y\odot x)\hookrightarrow
z=x\hookrightarrow(y\hookrightarrow z)$,

\item[$(3)$] \ $x\le y\Longleftrightarrow x\rightarrow
y=1\Longleftrightarrow x\hookrightarrow y=1$,

\item[$(4)$] \ $(x\hookrightarrow y)\hookrightarrow x\le
(x\hookrightarrow y)\rightarrow ((x\hookrightarrow
y)\hookrightarrow y)$,

\item[$(5)$] \ $x\le y\Rightarrow x\odot z\le y\odot z$,

\item[$(6)$] \ $x\le y\Rightarrow z\odot x\le z\odot y$,

\item[$(7)$] \ $x\odot y\le x$, \ $x\odot y\le y$,

\item[$(8)$] \ $x\odot 0=0\odot x=0$,

\item[$(9)$] \ $1\rightarrow x=1\hookrightarrow x=x$,

\item[$(10)$] \ $y\le x\rightarrow y$.
\end{enumerate}

A non-empty  subset $I$ of a pseudo $BL$-algebra $A$ is called a
{\it filter} of $A$ if it satisfies the following two conditions:

$(i)$ \ $x\odot y\in I$ for all $x,y\in I$;

$(ii)$ \ $x\le y \Longrightarrow y\in I$ for all $x\in I$ and
$y\in
A$.\\

Remind that a filter $I$ is called

{\it implicative } if \ $\left\{\begin{array}{ll}(x\rightarrow
y)\hookrightarrow x\in I \
{\rm implies } \ x\in I,\\
(x\hookrightarrow y)\rightarrow x\in I \ {\rm implies } \ x\in I,
\end{array}\right.$

{\it $MV$-filter} if \ $\left\{\begin{array}{ll} x\rightarrow y\in
I \ {\rm implies } \ ((y\rightarrow
x)\hookrightarrow x)\rightarrow y\in I,\\
x\hookrightarrow y\in I \ {\rm implies } \ ((y\hookrightarrow
x)\rightarrow x)\hookrightarrow y\in I,
\end{array}\right.$

{\it $G$-filter} if \ $\left\{\begin{array}{ll}
x\rightarrow(x\rightarrow y)\in I \ {\rm implies } \ x\rightarrow
y\in I,\\
x\hookrightarrow(x\hookrightarrow y)\in I \ {\rm implies } \
x\hookrightarrow y\in I.
\end{array}\right.$\\

Now, we introduce the concept of fuzzy (implicative) filters of
pseudo $BL$-algebras as follows:

\begin{definition}\label{D2.2}\rm A fuzzy set $\mu$ of a pseudo
$BL$-algebra $A$ is called a {\it fuzzy filter} of $A$ if

$(i)$ \ $\mu(x\odot y)\ge\min\{\mu(x), \mu(y)\}$,

$(ii)$ \ $x\le y\Longrightarrow \mu(x)\le\mu(y)$,

\noindent
is satisfied for all $x,y\in A.$
\end{definition}
\begin{definition}\label{D2.3}\rm A fuzzy filter $\mu$ of a pseudo
$BL$-algebra $A$ is called a {\it fuzzy implicative filter} if

$(iii)$ \ $\mu(x)\ge\max\{\mu((x\rightarrow y)\hookrightarrow x),
\mu((x\hookrightarrow y)\rightarrow x)$,\\
holds for all $x,y,z\in A$.
\end{definition}

For any fuzzy set $\mu$ of $A$ and $t\in (0,1]$, the set
$\mu_t=\{x\in A\,|\,\mu(x)\ge t\}$ is called a {\it level subset}
of $\mu$.

It is not difficult to verify that the following theorem is true.

\begin{theorem}\label{T2.4} A fuzzy set $\mu$ of a pseudo
$BL$-algebra $A$ is a fuzzy filter of $A$ if and only if each its
non-empty level subset is a filter of $A$. \hfill$\Box{}$
\end{theorem}

\paragraph {}
By an {\it interval number } $\widehat{a},$ we mean  an interval
$[a^{\bot},a^{\top}]$, where $0\le a^{\bot}\le a^{\top} \le 1$. The
set of all interval numbers is denoted by $D[0,1]$. The interval
$[a,a]$ can be simply identified with the number $a\in [0,1]$.

For the interval numbers $\widehat{a}_i=[a_i^{\bot},a_i^{\top}]$,
$\widehat{b}_i=[b_i^{\bot},b_i^{\top}] \in D[0,1]$, $i\in I$, we
define
\[
 \mathrm{rmax}\{\widehat{a}_i,\widehat{b}_i\}=[\max\{
a_i^{\bot},b_i^{\bot}\},\max\{a_i^{\top},b_i^{\top}\}],
\]
\[
\mathrm{rmin}\{\widehat{a}_i,\widehat{b}_i\}=[\min\{
a_i^{\bot},b_i^{\bot}\}, \min\{a_i^{\top},b_i^{\top}\}],
\]
\[
\mathrm{rinf}\,\widehat{a}_i=[\bigwedge\limits_{i\in I}
a_i^{\bot}, \bigwedge\limits_{i\in I} a_i^{\top}], \ \ \ \ \ \
\mathrm{rsup}\,\widehat{a}_i=[\bigvee\limits_{i\in I} a_i^{\bot},
\bigvee\limits_{i\in I} a_i^{\top}]
\]
and put

\begin{enumerate}
\item[(1)] \ $\widehat{a}_1\le \widehat{a}_2\Longleftrightarrow
a_1^{\bot}\le a_2^{\bot}$ and $a_1^{\top}\le a_2^{\top}$,

\item[(2)] \ $\widehat{a}_1=\widehat{a}_2\Longleftrightarrow  a_1^{\bot}=a_2^{\bot}$ and
$a_1^{\top}=a_2^{\top}$,

\item[(3)] \ $\widehat{a}_1<\widehat{a}_2\Longleftrightarrow
\widehat{a}_1\le \widehat{a}_2$ and $\widehat{a}_1\ne
\widehat{a}_2$,

\item[(4)] \ $k\widehat{a}=[ka^{\bot},ka^{\top}]$, whenever $ 0\le k\le 1$.
\end{enumerate}
Then, it is clear that $(D[0,1],\le,\vee,\wedge)$ is a complete
lattice with $0=[0,0]$ as its least element and $1=[1,1]$ as its
greatest element.

\medskip
The interval valued fuzzy sets provide a more adequate description
of uncertainty than the traditional fuzzy sets; it is therefore
important to use interval valued fuzzy sets in applications. One
of the main applications of fuzzy sets is fuzzy control, and one
of the most computationally intensive part of fuzzy control is the
''defuzzification''. Since a transition to interval valued fuzzy
sets usually increase the amount of computations, it is vitally
important to design faster algorithms for the corresponding
defuzzification. For more details, the reader can find some good
examples in \cite{7} and \cite{30}.

Recall that an {\it interval valued fuzzy set} $F$ on $X$ is the
set
 \[
F=\{(x, [\mu^{\bot}_F(x),\mu^{\top}_F(x)])\,|\,x\in X\},
\]
where $\mu^{\bot}_F $ and $\mu^{\top}_F$ are two fuzzy subsets of
$X$ such that $\mu^{\bot}_F(x)\le \mu^{\top}_F(x)$ for all $x\in
X.$ Putting
$\widehat{\mu_F}(x)=[\mu^{\bot}_F(x),\mu^{\top}_F(x)]$, we see
that $F=\{(x,\widehat{\mu_F}(x))\,|\,x\in X\}$, where
$\widehat{\mu_F}: X\rightarrow D[0,1].$

If $A$, $B$ are two interval valued fuzzy sets of $X$, then we
define

\smallskip
$A\subseteq B$ if and only if for all $x\in X$,
$\mu_A^{\bot}(x)\le \mu_B^{\bot}(x)$ and $\mu_A^{\top}(x)\le
\mu_B^{\top}(x)$,

\smallskip
$A=B$ if and only if for all $x\in X$, $\mu_A^{\bot}(x)=
\mu_B^{\bot}(x)$ and $\mu_A^{\top}(x)= \mu_B^{\top}(x)$.

\smallskip
Also, the union, intersection and complement are defined as follows:

\smallskip
$A\cup
B=\{(x,[\max\{\mu_A^{\bot}(x),\mu_B^{\bot}(x)\},\max\{\mu_A^{\top}(x),\mu_B^{\top}(x)\}]
\,|\,x\in X\},$

\smallskip
$A\cap
B=\{(x,[\min\{\mu_A^{\bot}(x),\mu_B^{\bot}(x)\},\min\{\mu_A^{\top}(x),\mu_B^{\top}(x)\}]
\,|\,x\in X\},$

\smallskip
$A^c=\{(x,[1-\mu_A^{\top}(x), 1-\mu_A^{\bot}(x)])\vert x\in X\}$,

\smallskip\noindent where $A^c$ is the {\it complement of interval valued
fuzzy set} $A$ in $X$.

\section{Interval valued $(\in,\ivq)$-fuzzy  filters}

Based on the results of \cite{1} and \cite{2}, we can extend the
concept of quasi-coincidence of fuzzy point within a fuzzy set to
the concept of quasi-coincidence of a fuzzy interval value with an
interval valued fuzzy set.

An interval valued fuzzy  set $F$ of a pseudo $BL$-algebra $A$ of
the form

$$\widehat{\mu_F}(y)=\left\{\begin{array}{ll} \widehat{\,t}\ne [0,0] & \mbox{\ \ \ \ \
if\ \ \ \ \ } y=x,\\ $[0,0]$ & \mbox{\ \ \ \ \ if\ \ \ \ \ } y\ne x,
\end{array}\right.$$
is said to be a {\it fuzzy interval value  with support $x$ and
interval value} $\widehat{\,t}$ and is denoted by
$U(x;\widehat{\,t})$. We say that a fuzzy interval value
$U(x;\widehat{\,t})$ {\it belongs to }(or resp. is {\it
quasi-coincident with}) an interval valued fuzzy set $F$, written
by $U(x;\widehat{\,t})\in F$ (resp. $U(x;\widehat{\,t}) q F$) if
$\widehat{\mu_F}(x)\ge \widehat{\,t}$ (resp.
$\widehat{\mu_F}(x)+\widehat{\,t}>[1,1])$. If
$U(x;\widehat{\,t})\in F$ or $U(x;\widehat{\,t})q F$, then we
write $U(x;\widehat{\,t})\ivq$. If $U(x;\widehat{\,t})\in F$ and
$U(x;\widehat{\,t})q F$, then we write
$U(x;\widehat{\,t})\in\wedge qF$. The symbol $\overline{\ivq}$
means that $\ivq$ does not hold.

In what follows, $A$ is a pseudo $BL$-algebra unless otherwise
specified. We emphasis that
$\widehat{\mu_F}(x)=[\mu^{\bot}_F(x),\mu^{\top}_F(x)]$ must
satisfy the following properties:
\medskip

$[\mu^{\bot}_F(x),\mu^{\top}_F(x)]<[0.5,0.5]$ or $[0.5,0.5]\le
[\mu^{\bot}_F(x),\mu^{\top}_F(x)]$, for all $x\in A$.
\medskip

First, we can extend  the  concept of fuzzy   filters to the
concept of interval valued fuzzy  filters of $A$  as follows:

\begin{definition}\rm  An interval valued  fuzzy set $F$  of $A$ is said to be an
{\it interval valued  fuzzy filter} of $A$ if  the following two
conditions hold:

$(F_1)$ $ \widehat{\mu_F}(x\odot y)  \ge
\mathrm{rmin}\{\widehat{\mu_F}(x),
  \widehat{\mu_F}(y)\}$ \ $\forall x,y\in A, $

$(F_2)$  $  x\le y\Rightarrow \widehat{\mu_F}(x)\le
\widehat{\mu_F}(y)$ \ $\forall x,y\in A.$
\end{definition}

Let $F$ be an interval valued fuzzy set.  Then, for every $t\in
(0,1]$, the set $F_{\widehat{\,t}}=\{x\in A\,|\,
\widehat{\mu_F}(x)\ge \widehat{\,t}\}$ is called the {\it level
subset of $F$}.

Now, we characterize the interval valued fuzzy  filters by using
their level  filters.

\begin{theorem}\label{T3.2} {\it An interval valued  fuzzy set  $F$ of $A$
is an interval valued  fuzzy filter of $A$ if and only if for any
$[0,0]<\widehat{\,t}\le [1,1]$ each non-empty $F_{\widehat{\,t}}$
is a filter of $A$.}
\end{theorem}
\begin{proof} The proof is similar to Theorem
\ref{T2.4}. \end{proof}

Further, we define the following concept:

\begin{definition}\rm An interval valued   fuzzy set $F$ of $A$ is said
to be an {\it  interval valued   $(\in,\ivq)$-fuzzy filter} of $A$
if for all $t,r\in (0,1]$ and $x,y\in A$,

\smallskip
$(F_3)$ \ $U(x; \widehat{\,t}\,)\in F$ and $U(y;\widehat{\,r} )\in
F$
 imply $U(x\odot y;\mathrm{ rmin}\{\widehat{\,t},\widehat{\,r}\})\ivq F$,

\smallskip
$(F_4)$ \ $U(x;\widehat{\,r})\in F$ implies
$U(y;\widehat{\,r})\ivq F$ with $x\le y$.
\end{definition}

\begin{example}\label{Ex3.4}\rm Let $I$ be a filter of a pseudo
$BL$-algebra $A$ and let $F$ be an interval valued fuzzy set in
$A$ defined by
$$
\widehat{\mu_F}(x)=\left\{\begin{array}{lllll} $[0.7,0.8]$ &
\mbox{if } x\in I,\\
$[0.3,0.4]$ & \mbox{otherwsie}.
\end{array}\right.
$$
It is easily to verify that $F$ is an interval valued
$(\in,\ivq)$-fuzzy filter of $A$.
\end{example}

\begin{theorem}\label{T3.5} An interval valued fuzzy set $F$ of $A$ is an interval valued
$(\in,\ivq)$-fuzzy filter if and only if for all $x,y\in A$ the
following two conditions are satisfied:

\smallskip
$(F_5)$ \ $\widehat{\mu_F}(x\odot y)\ge
\mathrm{rmin}\{\widehat{\mu_F}(x),\widehat{\mu_F}(y),\,0.5\}$,

\smallskip $(F_6)$ \ $x\le y\Longrightarrow
\widehat{\mu_F}(y)\ge\mathrm{rmin}\{\widehat{\mu_F}(x),\,0.5\}.$
\end{theorem}
\begin{proof} At first we prove that the conditions $(F_3)$ and
$(F_5)$ are equivalent.

Suppose that $(F_3)$ do not implies $(F_5)$, i.e., $(F_3)$ holds
but $(F_5)$ is not satisfied. In this case there are $x,y\in A$
such that
$$
\widehat{\mu_F}(x\odot y)<
\mathrm{rmin}\{\widehat{\mu_F}(x),\widehat{\mu_F}(y), [0.5,0.5]\}.
$$
If $\mathrm{rmin}\{\widehat{\mu_F}(x),\widehat{\mu_F}(y)\}<[0.5,
0.5]$, then $\widehat{\mu_F}(x\odot y)<
\mathrm{rmin}\{\widehat{\mu_F}(x),\widehat{\mu_F}(y)\}$. This
means that for some $t $ satisfying the condition
$\widehat{\mu_F}(x\odot y)<\widehat{\,t}<
\mathrm{rmin}\{\widehat{\mu_F}(x),\widehat{\mu_F}(y)\}$, we have
$U(x;\widehat{\,t}\,)\in F$ and $U(y;\widehat{\,t}\,)\in F$, but
$U(x\odot y;\widehat{\,t}\,)\overline{\ivq} F$, which contradicts
to $(F_3)$. So, this case is impossible. Therefore
$\mathrm{rmin}\{\widehat{\mu_F}(x),\widehat{\mu_F}(y)\}\geq
[0.5,0.5]$. In this case $\widehat{\mu_F}(x\odot y)<[0.5,0.5]$,
$U(x;[0.5,0.5])\in F$, $U(y;[0.5,0.5])\in F$ and $U(x\odot y;[0.5,
0.5 ])\overline{\ivq}F$, which also is impossible. So, $(F_3)$
implies $(F_5)$.

Conversely, if $(F_5)$ holds and $U(x;\widehat{\,t}\,)\in F$,
$U(y;\widehat{r})\in F$, then $\widehat{\mu_F}(x)\ge
\widehat{\,t}$, $\widehat{\mu_F}(y)\ge \widehat{r}$ and
$\widehat{\mu_F}(x\odot y)\ge
\mathrm{rmin}\{\widehat{\,t},\widehat{r},[0.5,0.5]\}$. If
$\mathrm{rmin}\{\widehat{\,t},\widehat{r}\}>[0.5,0.5]$, then
$\widehat{\mu_F}(x\odot y)\ge [0.5,0.5]$, which implies $F(x\odot
y)+\mathrm{rmin}\{\widehat{\,t},\widehat{r}\,\}>[1,1]$, i.e.,
$U(x\odot y;\mathrm{rmin}\{\widehat{\,t},\widehat{r}\,\})q F$. If
$\mathrm{rmin}\{\widehat{\,t},\widehat{r}\,\}\le [0.5,0.5]$, then
$\widehat{\mu_F}(x\odot y)\ge
\mathrm{rmin}\{\widehat{\,t},\widehat{r}\,\}$. Thus $U(x\odot
y;\mathrm{rmin}\{\widehat{\,t},\widehat{r}\,\})\in F$. Therefore,
$U(x\oplus y;\mathrm{rmin}\{\widehat{\,t},\widehat{r}\,\})\ivq F$.
Summarizing, $(F_5)$ implies $(F_3)$. So $(F_3)$ and $(F_5)$ are
equivalent.

To prove that $(F_4)$ and $(F_6)$ are equivalent suppose that
$(F_6)$ is not satisfied, i.e.,
$\widehat{\mu_F}(y_0)<\mathrm{rmin}\{\widehat{\mu_F}(x_0),\,[0.5,0.5]\}$
for some $x_0\le y_0$. If $\widehat{\mu_F}(x_0)<[0.5,0.5]$, then
$\widehat{\mu_F}(y_0)<\widehat{\mu_F}(x_0)$, which means that
there exists $s $ such that
$\widehat{\mu_F}(y_0)<\widehat{s}<\widehat{\mu_F}(x_0)$ and
$\widehat{\mu_F}(y_0)+\widehat{\mu_F}(x_0)<[1,1]$. Thus
$U(y_0;\widehat{\,s}\,)\in F$ and
$U(x_0;\widehat{\,s}\,)\overline{\ivq}F$, which contradicts to
$(F_4)$. So, $\widehat{\mu_F}(x_0)\geq [0.5,0.5]$. But in this
case
$\widehat{\mu_F}(y_0)<\mathrm{rmin}\{\widehat{\mu_F}(x_0),\,[0.5,0.5]\}$
gives $U(x_0;[0.5,0.5])\in F$ and
$U(x_0;[0.5,0.5])\overline{\in\vee q}F$, which contradicts to
$(F_4)$. Hence
$\widehat{\mu_F}(y)\geq\mathrm{rmin}\{\widehat{\mu_F}(x),\,[0.5,0.5]\}$
for all $x\leq y$, i.e., $(F_4)$ implies $(F_6)$.

Conversely, if $(F_6)$ holds, then $x\le y$ and
$U(x;\widehat{\,t}\,)\in F$ imply $\widehat{\mu_F}(x)\ge
\widehat{\,t}$, and so
$\widehat{\mu_F}(y)\ge\mathrm{rmin}\{\widehat{\mu_F}(x),\,[0.5,0.5]\}\ge
\mathrm{rmin}\{\widehat{\,t},[0.5,0.5]\}$. Thus
$\widehat{\mu_F}(y)\ge \widehat{\,t}$ or $\widehat{\mu_F}(y)\ge
[0.5,0.5],$ according to $\widehat{\,t}\le [0.5,0.5]$ or
$\widehat{\,t}>[0.5,0.5]$. Therefore, $U(y;\widehat{\,t}\,)\ivq
F$. Hence $(F_6)$ implies $(F_4)$.
\end{proof}

\begin{proposition}\label{P3.6} An interval valued fuzzy set $F$ of $A$ is an interval valued  $(\in,\ivq)$-fuzzy
filter if and only if

\smallskip
$(F_7)$ \ $\widehat{\mu_F}(1)\ge
\mathrm{rmin}\{\widehat{\mu_F}(x),\,[0.5,0.5]\}$ holds for all
$x\in A$

\smallskip\noindent and one of the conditions:

\smallskip
$(F_8)$ \
$\widehat{\mu_F}(y)\ge\mathrm{rmin}\{\widehat{\mu_F}(x),\,\widehat{\mu_F}(x\rightarrow
y),\,[0.5,0.5]\}$,

\smallskip
$(F_8^{\prime})$ \  $\widehat{\mu_F}(y)\ge \mathrm{rmin}\{\widehat{\mu_F}(x),\,\widehat{\mu_F}(x\hookrightarrow y),
[0.5,0.5]\}$

\smallskip\noindent is satisfied for all $x,y\in A$.
\end{proposition}
\begin{proof} The proof is similar to the proof of Proposition 4.7
from the first part of \cite{8}, so we omit it.
\end{proof}

Now, we characterize the interval valued   $(\in,\ivq )$-fuzzy
filters by using their level subsets.

\begin{theorem}\label{T3.7} An interval valued   fuzzy set $F$ of $A$ is an
interval valued $(\in,\ivq)$-fuzzy filter if and only if for all
$[0,0]<\widehat{\,t}\le [0.5,0.5]$ all nonempty level subsets
$F_{\widehat{t}}$ are filters of $A$.
\end{theorem}
\begin{proof} Let $F$ be an interval valued $(\in,\ivq)$-fuzzy
filter of $A$ and $ [0,0] <\widehat{\,t}\le [0.5,0.5]$. If $x,y\in
F_{\widehat{t}}$, then $\widehat{\mu_F}(x)\ge \widehat{\,t}$ and
$\widehat{\mu_F}(y)\ge \widehat{\,t}$. Now we have
$\widehat{\mu_F}(x\odot y) \ge \mathrm{ rmin}\{\widehat{\mu_F}(x),
\widehat{\mu_F}(y),[0.5,0.5]\}$ $\ge \mathrm{rmin}\{\widehat{\,t},
[0.5, 0.5]\}=\widehat{\,t}$. This means that $x\odot y\in
F_{\widehat{t}}$.  Let $x,y\in A$ be such that $x\le y$. If $x\in
F_{\widehat{t}}$, then, by $(F_4)$, we have $\widehat{\mu_F}(y)\ge
\mathrm{ rmin}\{\widehat{\mu_F}(x),[0.5,0.5]\}\ge \mathrm{
rmin}\{\widehat{\,t}, [0.5,0.5]\}=\widehat{\,t}$, which implies
$y\in F_{\widehat{t}}$. Hence, $F_{\widehat{t}}$ is a filter of
$A$.

Conversely, let $F$ be an interval valued fuzzy set of $A$ such
that all nonempty $F_{\widehat{t}}$, where $ [0,0]
<\widehat{\,t}\le [0.5,0.5]$, are filters of $A$. Then, for every
$x,y\in A$, we have

$\widehat{\mu_F}(x)\ge \mathrm{ rmin}\{\widehat{\mu_F}(x),
\widehat{\mu_F}(y),[0.5,0.5]\} = \widehat{t}_0$,

\smallskip
$\widehat{\mu_F}(y)\ge \mathrm{
rmin}\{\widehat{\mu_F}(x),\widehat{\mu_F}(y),[0.5,0.5]\}=
\widehat{t}_0 $.

\smallskip\noindent
Thus, $x,y\in F_{\widehat{t}_0}$, and so $x\odot y\in
F_{\widehat{t}_0}$, i.e., $\widehat{\mu_F}(x\odot
y)\ge\mathrm{rmin}\{\widehat{\mu_F}(x),
\widehat{\mu_F}(y),[0.5,0.5]\}$. If $x,y\in A$ and $x\le y$, then
$\widehat{\mu_F}(x)\ge \mathrm{
rmin}\{\widehat{\mu_F}(x),[0.5,0.5]\}= \widehat{s}_0$. Hence $x\in
F_{\widehat{s}_0}$, and so $y\in F_{\widehat{s}_0}$. Thus
$\widehat{\mu_F}(y)\ge \widehat{s}_0=\mathrm{
rmin}\{\widehat{\mu_F}(x), [0.5,0.5]\}$. Therefore, $F$ is an
interval valued $(\in,\ivq)$-fuzzy filter of $A$.
\end{proof}

Naturally, we can establish a similar result when each nonempty
$F_{\widehat{t}}$ is a filter of $A$ for
$[0.5,0.5]<\widehat{\,t}\le [1,1]$.

\begin{theorem}\label{T3.8} For $[0.5,0.5]<\widehat{\,t}\le [1,1]$ each nonempty level subset
$F_{\widehat{t}}$ of an interval valued fuzzy set $F$ of $A$ is a
filter if and only if for all $x,y\in A$ the following two
conditions are satisfied:

\smallskip
$(F_9)$ \ $\mathrm{rmax}\{\widehat{\mu_F}(x\odot
y),\,[0.5,0.5]\}\ge\mathrm{rmin}\{\widehat{\mu_F}(x),\widehat{\mu_F}(y)\}$,

\smallskip
$(F_{10})$ \ $\mathrm{rmax}\{\widehat{\mu_F}(y),\,[0.5,0.5]\}\ge
\widehat{\mu_F}(x)$ when $x\le y$.
\end{theorem}
\begin{proof} Let $F_{\widehat{t}}$ be a nonempty
level subset of $F$. Assume that $F_{\widehat{t}}$ is a filter of
$A$. If $\mathrm{rmax}\{\widehat{\mu_F}(x\odot
y),\,[0.5,0.5]\}<\mathrm{rmin}\{\widehat{\mu_F}(x),\widehat{\mu_F}(y)\}=\widehat{\,t}$
for some $x,y\in A$, then $[0.5,0.5]<\widehat{\,t}\le [1,1]$,
$\widehat{\mu_F}(x\odot y)<\widehat{\,t}$ and $x,y\in
F_{\widehat{t}}\,$. Thus $x\odot y\in F_{\widehat{t}}$, whence
$\widehat{\mu_F}(x\odot y)\ge \widehat{\,t}$, which contradicts to
$\widehat{\mu_F}(x\odot y)<\widehat{\,t}$. So, $(F_9)$ is
satisfied.

If there exist $x,y\in A$ such that
$\mathrm{rmax}\{\widehat{\mu_F}(y),\,[0.5,0.5]\}
<\widehat{\mu_F}(x)=\widehat{\,t}$, then
$[0.5,0.5]<\widehat{\,t}\le [1,1]$,
$\widehat{\mu_F}(y)<\widehat{\,t}$ and $x\in F_{\widehat{\,t}}$.
Since $x\leq y$ we also have $y\in F_{\widehat{t}}\,$. Thus
$\widehat{\mu_F}(y)\ge \widehat{\,t}$, which is impossible.
Therefore $\mathrm{rmax}\{\widehat{\mu_F}(y),\,[0.5,0.5]\}\ge
\widehat{\mu_F}(x)$ for $x\le y$.

Conversely, suppose that the conditions $(F_9)$ and $(F_{10})$ are
satisfied. In order to prove that for $[0.5,0.5]<\widehat{\,t}\le
[1,1]$ each nonempty level subset $F_{\widehat{t}}$ is a filter of
$A$ assume that $x,y\in F_{\widehat{t}}$. In this case
$[0.5,0.5]<\widehat{\,t}\le\mathrm{rmin}\{\widehat{\mu_F}(x),
\widehat{\mu_F}(y)\}\le\mathrm{rmax}\{\widehat{\mu_F}(x\odot
y),\,[0.5,0.5]\}=\widehat{\mu_F}(x\odot y)$, which proves $x\odot
y\in F_{\widehat{t}}$. If $x\le y$ and $x\in F_{\widehat{t}}$,
then $[0.5,0.5]<\widehat{\,t}\le \widehat{\mu_F}(x)\le
\mathrm{rmax}\{\widehat{\mu_F}(y),\,[0.5,0.5]\}=\widehat{\mu_F}(y)$,
and so $y\in F_{\widehat{\,t}}$. This completes the
proof.
\end{proof}

Let $J=\{t\in (0,1]\,|\,F_{\widehat{t}}\ne\emptyset\}$, where $F$
is an interval valued fuzzy set of $A$. For $J=(0,1]$ \ $F$ is an
ordinary interval valued fuzzy filter of $A$ (Theorem \ref{T3.2});
for $J=(0, 0.5 ]$ it is an $(\in,\ivq)$-fuzzy filter of $A$
(Theorem \ref{T3.7}).

In \cite{26}, Yuan, Zhang and Ren  gave the definition of a fuzzy
subgroup with thresholds which is a generalization of Rosenfeld's
fuzzy subgroup, and also Bhakat and Das's fuzzy subgroup. Based on
the results of \cite{26}, we can extend the concept of a fuzzy
subgroup with thresholds to the concept of an interval valued  fuzzy
filter with thresholds in the following way:

\begin{definition}\label{D3.9}\rm Let $[0,0]\le\widehat{\alpha}<\widehat{\beta}\leq [1,1]$. An
interval valued  fuzzy set $F$ of $A$ is called an {\it interval
valued  fuzzy filter with thresholds}
$(\widehat{\alpha},\widehat{\beta})$ if for all $x,y\in A$, the
following two conditions are satisfied:

\smallskip
$(F_{11})$ \ $\mathrm{rmax}\{\widehat{\mu_F}(x\odot
y),\widehat{\alpha}\}\ge\mathrm{rmin}\{\widehat{\mu_F}(x),\widehat{\mu_F}(y),
\widehat{\beta}\}$,

\smallskip
$(F_{12})$ \
$\mathrm{rmax}\{\widehat{\mu_F}(y),\widehat{\alpha}\}\ge\mathrm{rmin}
\{\widehat{\mu_F}(x),\widehat{\beta}\}$
for $x\le y$.
\end{definition}

\begin{theorem}\label{T3.10} An interval valued  fuzzy set $F$ of $A$ is
an interval valued  fuzzy filter with thresholds
$(\widehat{\alpha},\widehat{\beta})$ if and only if each nonempty
$F_{\widehat{t}}$, where
$\widehat{\alpha}<\widehat{\,t}\le\widehat{\beta}$ is a filter of
$A$.
\end{theorem}
\begin{proof} The proof is similar to the proof of Theorems
\ref{T3.7} and \ref{T3.8}.
\end{proof}

\section{Interval valued $(\in,\ivq)$-fuzzy implicative filters}

\begin{definition}\label{D4.1}\rm An interval valued $(\in,\ivq)$-fuzzy filter $F$ of
$A$ is called an {\it interval valued $(\in,\ivq)$-fuzzy implicative
filter} if for all $x,y\in A$ it satisfies the condition:

\medskip
$(F_{13})$ \ $\left\{\begin{array}{llll}
\widehat{\mu_F}(x)\ge\mathrm{rmin}\{\widehat{\mu_F}((x\rightarrow y)\hookrightarrow x),
 \,[0.5,0.5]\},\\[1mm]
\widehat{\mu_F}(x)\ge\mathrm{rmin}\{\widehat{\mu_F}((x\hookrightarrow
y)\rightarrow x),\,[0.5,0.5]\}.
\end{array}\right.$
\end{definition}

\medskip
The following proposition is obvious.

\begin{proposition}\label{P4.2} If $F$ is an interval valued $(\in,\ivq)$-fuzzy
implicative filter of $A$, then

\medskip
$(1)$ \ $\left\{\begin{array}{llll}\widehat{\mu_F}(x)\ge
\mathrm{rmin}\{\widehat{\mu_F}((x\rightarrow y)\rightarrow x),\,[0.5,0.5]\},\\[1mm]
\widehat{\mu_F}(x)\ge
\mathrm{rmin}\{\widehat{\mu_F}((x\hookrightarrow y)\hookrightarrow
x),\,[0.5,0.5]\},
\end{array}\right.$

\medskip
$(2)$ \ $\left\{\begin{array}{llll} \widehat{\mu_F}(((y\rightarrow
x)\hookrightarrow x)\rightarrow
y)\ge\mathrm{rmin}\{\widehat{\mu_F}(x\rightarrow
y),\,[0.5,0.5]\},\\[1mm]
\widehat{\mu_F}(((y\hookrightarrow x)\rightarrow x)\hookrightarrow
y)\ge \mathrm{rmin}\{\widehat{\mu_F}(x\hookrightarrow
y),\,[0.5,0.5]\},
\end{array}\right.$

\medskip
$(3)$ \ $\left\{\begin{array}{llll}
\widehat{\mu_F}((y\hookrightarrow x)\rightarrow x)\ge
\mathrm{rmin}\{\widehat{\mu_F}((x\rightarrow y)\hookrightarrow
y),\,[0.5,0.5]\},\\[1mm]
\widehat{\mu_F}((y\rightarrow x)\hookrightarrow x)\ge
\mathrm{rmin}\{\widehat{\mu_F}((x\hookrightarrow y)\rightarrow
y),\,[0.5,0.5]\},
\end{array}\right.$

\medskip
$(4)$ \ $\left\{\begin{array}{llll}\widehat{\mu_F}((y\rightarrow
x)\hookrightarrow x)\ge \mathrm{rmin}\{\widehat{\mu_F}((x\rightarrow
y)\hookrightarrow
y),\,[0.5,0.5]\},\\[1mm]
\widehat{\mu_F}((y\hookrightarrow x)\rightarrow x)\ge
\mathrm{rmin}\{\widehat{\mu_F}((x\hookrightarrow y)\hookrightarrow
y),\,[0.5,0.5]\}
\end{array}\right.$\\[1mm]
hold for all $x,y\in A$. \hfill$\Box{}$
\end{proposition}

\begin{theorem}\label{T4.3}
An interval valued  fuzzy set $F$ of $A$ is an interval valued
$(\in,\ivq)$-fuzzy implicative filter if and only if for
$[0,0]<\widehat{\,t}\le [0.5,0.5]$ each nonempty level subset
$F_{\widehat{t}}$ is an implicative filter of $A$.
\end{theorem}
\begin{proof} Let $F$ be an interval valued $(\in,\ivq)$-fuzzy
implicative filter of $A$ and $ [0,0]<\widehat{\,t}\le [0.5,0.5]$.
Then, by Theorem \ref{T3.7}, each nonempty $F_{\widehat{t}}$ is a
filter of $A$. For all $x,y\in A$ from $(x\rightarrow
y)\hookrightarrow x\in F_{\widehat{t}}$ it follows
$\widehat{\mu_F}((x\rightarrow y)\hookrightarrow x)\ge
\widehat{\,t}$. This, according to $(F_{13})$, gives
$\widehat{\mu_F}(x)\ge
\mathrm{rmin}\{\widehat{\mu_F}((x\rightarrow y)\hookrightarrow
x),\,[0.5,0.5]\}\ge
\mathrm{rmin}\{\widehat{\,t},[0.5,0.5]\}=\widehat{\,t}$. So, $x\in
F_{\widehat{t}}$. Analogously $(x\hookrightarrow y)\rightarrow
x\in F_{\widehat{t}}$ implies $x\in F_{\widehat{t}}$. Therefore
$F_{\widehat{t}}$ is an implicative filter of $A$.

Conversely, if $F$ is an interval valued  fuzzy set of $A$ such
that for $[0,0]<\widehat{t}\le [0.5,0.5]$ each nonempty level set
$F_{\widehat{t}}$ is an implicative filter of $A$, then, by
Theorem \ref{T3.7}, $F$ is an interval valued $(\in,\ivq)$-fuzzy
filter of $A$. Putting $\widehat{\mu_F}((x\rightarrow
y)\hookrightarrow x)\ge
\widehat{s}_0=\mathrm{rmin}\{\widehat{\mu_F}((x\rightarrow
y)\hookrightarrow x),\,[0.5,0.5]\}$, we obtain $(x\rightarrow
y)\hookrightarrow x\in F_{\widehat{s}_0}$. Consequently, $x\in
F_{\widehat{s}_0}$, i.e., $\widehat{\mu_F}(x)\ge
\widehat{s}_0=\mathrm{rmin}\{\widehat{\mu_F}((x\rightarrow
y)\hookrightarrow x),\,[0.5,0.5]\}$. This proves the first
condition of $(F_{13})$. Similarly, we prove the second
condition.
 \end{proof}

Basing on our Theorem \ref{T3.8} the above result can be extended
to the case $[0.5,0.5]<\widehat{\,t}\le [1,1]$ in the following
way:

\begin{theorem}\label{T4.4} For $[0.5,0.5]<\widehat{\,t}\le [1,1]$ each nonempty level
subset $F_{\widehat{t}}$ of an interval valued fuzzy set $F$ of
$A$ is an implicative filter of $A$ if and only if the conditions
$(F_9)$, $(F_{10})$ and

\smallskip
$(F_{14})$ \
$\left\{\begin{array}{llll}\mathrm{rmax}\{\widehat{\mu_F}(x),\,[0.5,0.5]\}\ge
\widehat{\mu_F}((x\rightarrow
y)\hookrightarrow x),\\[1mm]
\mathrm{rmax}\{\widehat{\mu_F}(x), \,[0.5,0.5]\}\ge
\widehat{\mu_F}((x\hookrightarrow y )\rightarrow
x).\end{array}\right.$\\[1mm]
are satisfied for all $x,y\in A$.
\end{theorem}
\begin{proof} According to Theorem \ref{T3.8}, for
$[0.5,0.5]<\widehat{\,t}\le [1,1]$ each nonempty level subset
$F_{\widehat{t}}$ of $F$ is a filter of $A$ if and only if $F$
satisfies $(F_9)$ and $(F_{10})$. So, we shall prove only that a
filter $F_{\widehat{t}}$ is implicative if and only if $F$
satisfies $(F_{14})$.

To prove $(F_{14})$ suppose the existence of $x,y\in A$ such that
$$
\mathrm{rmax}\{\widehat{\mu_F}(x),\,[0.5,0.5]\}<\widehat{\,t}=
\widehat{\mu_F}((x\rightarrow y)\hookrightarrow x).
$$
In this case $[0.5,0.5]<\widehat{\,t}\le [1,1]$,
$\widehat{\mu_F}(x)<\widehat{\,t}$ and $(x\rightarrow
y)\hookrightarrow x\in F_{\widehat{\,t}}.$ Since $F_{\widehat{t}}$
is an implicative filter of $A$, we have $x\in F_{\widehat{t}}$,
and so $\widehat{\mu_F}(x)\ge \widehat{\,t}$, which is a
contradiction. Similarly, we can prove the second inequality of
$(F_{12})$.

Conversely, suppose that an interval valued  fuzzy set $F$
satisfies $(F_{14})$ and each nonempty $F_{\widehat{t}}$ is a
filter of $A$. If $[0.5,0.5]<\widehat{\,t}\le [1,1]$ and
$(x\rightarrow y)\hookrightarrow x\in F_{\widehat{t}}$, then
$[0.5,0.5]<\widehat{\,t}\le
\mathrm{rmin}\{\widehat{\mu_F}((x\rightarrow y)\hookrightarrow
x)\le\mathrm{rmax}\{\widehat{\mu_F}(x),
[0.5,0.5]\}<\widehat{\mu_F}(x),$ which implies $x\in
F_{\widehat{\,t}}$. Similarly, from $(x\hookrightarrow
y)\rightarrow x\in F_{\widehat{\,t}}$ it follows $x\in
F_{\widehat{t}}$. Thus, $F_{\widehat{t}}$ is an implicative filter
of $A$.
 \end{proof}

Basing on the method presented in \cite{26}, we can extend the
concept of a fuzzy subgroup with thresholds to the concept of an
interval valued fuzzy implicative filter with thresholds.

\begin{definition}\label{D4.5}\rm Let $[0,0]\le \widehat{\alpha}<\widehat{\beta}\le [1,1]$.
An interval valued  fuzzy set $F$ of $A$ is called an {\it interval
valued  fuzzy implicative filter with thresholds}
$(\widehat{\alpha},\widehat{\beta})$ of $A$ if for all $x,y\in A$ it
satisfies $(F_{11})$, $(F_{12})$ and

\smallskip
$(F_{15})$ \
$\left\{\begin{array}{llll}\mathrm{rmax}\{\widehat{\mu_F}(x),\widehat{\alpha}\}\ge
\mathrm{rmin}\{\widehat{\mu_F}((x\rightarrow y)\hookrightarrow
x),\widehat{\beta}\},\\[1mm]
\mathrm{rmax}\{\widehat{\mu_F}(x),\widehat{\alpha}\}\ge\mathrm{rmin}\{\widehat{\mu_F}((x\hookrightarrow
y)\rightarrow x),\widehat{\beta}\}. \end{array}\right.$
\end{definition}

\begin{theorem}\label{T4.6} An interval valued  fuzzy set $F$ of $A$ is
an interval valued fuzzy implicative filter with thresholds
$(\widehat{\alpha},\widehat{\beta})$ if and only if each nonempty
$F_{\widehat{t}}$, where $\widehat{\alpha }<\widehat{\,t}\le
\widehat{\beta}$ is an implicative filter of $A$.
\end{theorem}
\begin{proof} The proof is similar to the proof of Theorems
\ref{T4.3} and \ref{T4.4}.
\end{proof}

\begin{definition}\label{D4.7}\rm An interval valued $(\in,\ivq)$-fuzzy filter of
$A$ is called an {\it interval valued $(\in,\ivq)$-fuzzy
$MV$-filter} of $A$ if

\smallskip
 $(F_{16})$ \
$\left\{\begin{array}{llll}\widehat{\mu_F}(((y\rightarrow
x)\hookrightarrow x)\rightarrow y)\ge
\mathrm{rmin}\{\widehat{\mu_F}(x\rightarrow
y),\,[0.5,0.5]\},\\[1mm]
\widehat{\mu_F}(((y\hookrightarrow x)\rightarrow x)\hookrightarrow
y)\ge
\mathrm{rmin}\{\widehat{\mu_F}(x\hookrightarrow
y),\,[0.5,0.5]\}\end{array}\right.$\\[2mm]
holds for all $x,y\in A$.
\end{definition}

It follows from Proposition \ref{P4.2}$(2)$ that every interval
valued $(\in,\ivq)$-fuzzy implicative filter is an interval valued
$(\in,\ivq)$-fuzzy $MV$-filter.

\begin{definition}\label{D4.8}\rm An interval valued $(\in,\ivq)$-fuzzy filter of $A$ is
called an {\it interval valued $(\in,\ivq)$-fuzzy $G$-filter} of
$A$ if
\smallskip

$(F_{17})$ \ $\left\{\begin{array}{llll}\widehat{\mu_F}(x\rightarrow
y) \ge
\mathrm{rmin}\{\widehat{\mu_F}(x\rightarrow (x\rightarrow
y)),\,[0.5,0.5]\},\\[1mm]
\widehat{\mu_F}(x\hookrightarrow y) \ge
\mathrm{rmin}\{\widehat{\mu_F}(x\hookrightarrow
(x\hookrightarrow y)),\,[0.5,0.5]\}\end{array}\right.$\\[2mm]
holds for all $x,y\in A$.
\end{definition}

\begin{lemma}\label{L4.9} Every interval valued $(\in,\ivq)$-fuzzy implicative
filter is an interval valued $(\in,\ivq)$-fuzzy $G$-filter.
\end{lemma}
\begin{proof} Let $F$ be an interval valued
$(\in,\ivq)$-fuzzy implicative filter of $A$. As it is well know
(see \cite{8}), for any $x,y\in A$, we have
$x\hookrightarrow(x\hookrightarrow y)\le ((x\hookrightarrow
y)\hookrightarrow y)\rightarrow (x\hookrightarrow y)$. Whence,
according to $(F_6)$ we obtain $\widehat{\mu_F}(((x\hookrightarrow
y)\hookrightarrow y)\rightarrow(x\hookrightarrow y))\ge
\mathrm{rmin}\{\widehat{\mu_F}(x\hookrightarrow(x\hookrightarrow
y)),\,[0.5,0.5]\}$. From this, applying $(F_{13})$, we get
$\widehat{\mu_F}(x\hookrightarrow y)\ge
\mathrm{rmin}\{\widehat{\mu_F}(((x\hookrightarrow
y)\hookrightarrow y)\rightarrow(x\hookrightarrow
y)),\,[0.5,0.5]\}\ge
\mathrm{rmin}\{\widehat{\mu_F}(x\hookrightarrow(x\hookrightarrow
y)),\,[0.5,0.5]\}$. This proves the second inequality of
$(F_{17})$.

The proof of the first inequality is similar.
 \end{proof}

\begin{theorem}\label{T4.10} For any interval valued $(\in,\ivq)$-fuzzy
filter $F$ of $A$ satisfying the identity
$\widehat{\mu_F}(x\rightarrow y)=\widehat{\mu_F}(x\hookrightarrow
y)$ the following two conditions are equivalent:

\begin{enumerate}
\item[{$(1)$}] $F$ is an interval valued $(\in,\ivq)$-fuzzy
implicative filter;
\item[{$(2)$}] $F$ is an interval valued $(\in,\ivq)$-fuzzy $MV$-filter and
an interval valued $(\in,\ivq)$-fuzzy $G$-filter.
\end{enumerate}
\end{theorem}
\begin{proof} $(1)\Longrightarrow (2)$ By Proposition
\ref{P4.2} $(2)$ and Lemma \ref{L4.9}.

$(2)\Longrightarrow (1)$ Let $F$ be an interval valued
$(\in,\ivq)$-fuzzy $MV$-filter and an interval valued
$(\in,\ivq)$-fuzzy $G$-filter of $A$. From $(x\hookrightarrow
y)\hookrightarrow x\le (x\hookrightarrow y)\rightarrow
((x\hookrightarrow y)\hookrightarrow y)$ (see \cite{8}), we have
\[\arraycolsep=.5mm
\begin{array}{lll}\widehat{\mu_F}((x\hookrightarrow y)\hookrightarrow
((x\hookrightarrow y)\hookrightarrow y))&
=\widehat{\mu_F}((x\hookrightarrow y)\rightarrow
((x\hookrightarrow y)\hookrightarrow
y))\\[2mm]
&\ge\mathrm{rmin}\{\widehat{\mu_F}((x\hookrightarrow
y)\hookrightarrow x),\,[0.5,0.5]\},
\end{array}
\]
which together with the fact that $F$ is an interval valued
$(\in,\ivq)$-fuzzy $G$-filter of $A$ gives
\[\arraycolsep=.5mm
\begin{array}{lll}\widehat{\mu_F}((x\hookrightarrow y)\hookrightarrow y)&\ge&\mathrm{rmin}
\{\widehat{\mu_F}((x\hookrightarrow y)\hookrightarrow ((x\rightarrow
y)\hookrightarrow y)),\,[0.5,0.5]\}\\[2mm]
&\ge&\mathrm{rmin}\{\widehat{\mu_F}((x\hookrightarrow
y)\hookrightarrow x),\,[0.5,0.5]\}.
\end{array}
\]
Moreover, from $y\le x\rightarrow y$ we get $(x\hookrightarrow
y)\hookrightarrow x\le y\hookrightarrow x$, and consequently
$$
\widehat{\mu_F}(y\hookrightarrow x)\ge
\mathrm{rmin}\{F((x\hookrightarrow y)\hookrightarrow
x),\,[0.5,0.5]\}.
$$
The fact that $F$ is an interval valued $(\in,\ivq)$-fuzzy
$MV$-filter of $A$ implies
\[\arraycolsep=.5mm
\begin{array}{lll}
\widehat{\mu_F}(((x\hookrightarrow y)\rightarrow y)\hookrightarrow
x)&\ge&\mathrm{rmin}\{\widehat{\mu_F}(y\hookrightarrow
x),\,[0.5,0.5]\}\\[2mm]
&\ge&\mathrm{rmin}\{\widehat{\mu_F}((x\hookrightarrow
y)\hookrightarrow x),\,[0.5,0.5]\}. \end{array}
\]
Since $F$ is an interval valued $(\in,\ivq)$-fuzzy filter of $A$, we
also have
$$
\widehat{\mu_F}(x)\ge
\mathrm{rmin}\{\widehat{\mu_F}(((x\hookrightarrow y)\rightarrow
y)\hookrightarrow x),\,\widehat{\mu_F}((x\hookrightarrow
y)\rightarrow y),\,[0.5,0.5]\}.
$$

Summarizing the above, we obtain $\widehat{\mu_F}(x)\ge
\mathrm{rmin}\{\widehat{\mu_F}((x\hookrightarrow y)\hookrightarrow
x),\,[0.5,0.5]\}$. Hence, $F$ is an interval valued
$(\in,\ivq)$-fuzzy implicative filter of $A$.
 \end{proof}

\paragraph {}\noindent{\bf Problem.} Prove or disprove that in Theorem
\ref{T4.10} the assumption $\widehat{\mu_F}(x\rightarrow
y)=\widehat{\mu_F}(x\hookrightarrow y)$ is essential.

\section{Implication-based interval valued fuzzy implicative filters}

Fuzzy logic is an extension of set theoretic variables to terms of
the linguistic variable truth. Some operators, like $\wedge,
\vee,\neg,\rightarrow$ in fuzzy logic also can be defined by using
the tables of valuations. Also, the extension principle can be
used to derive definitions of the operators.

In the fuzzy logic, the truth value of fuzzy proposition $P$ is
denoted by $[P]$. The correspondence between fuzzy logical and
set-theoretical notations is presented below:

$[x\in F]=F(x)$,

$[x\notin F]=1-F(x)$,

$[P\wedge Q]=\min\{[P],[Q]\}$,

 $[P\vee Q]=\max\{[P],[Q]\}$,

$[P\rightarrow Q]=\min\{1,1-[P]+[Q]\}$,

$[\forall xP(x)]=\inf [P(x)]$,

$\models P\Longleftrightarrow [P]=1$ for all valuations.

\medskip

Of course, various implication operators can be defined similarly.
In the table presented below we give the example of such
definitions. In this table $\alpha$ denotes the degree of truth
(or degree of membership) of the premise, $\beta$ is the values
for the consequence, and $I$ the result for the corresponding
implication:

\begin{tabular}{ll}\\ \hline
 Name & Definition of Implication Operators\\ \hline\\
Early Zadeh &
$I_m(\alpha,\beta)=\max\{1-\alpha,\min\{\alpha,\beta\}\}$,\\[2mm]
Lukasiewicz & $I_a(\alpha,\beta)=\min\{1,1-\alpha+\beta\}$,\\
Standard Star (G\"odel) &
$I_g(\alpha,\beta)=\left\{\begin{array}{lll} 1 & \mbox{if }&
\alpha\le\beta ,\\ \beta & \mbox{if }& \alpha>\beta ,\end{array}\right.$\\
Contraposition of G\"odel &
$I_{cg}(\alpha,\beta)=\left\{\begin{array}{lll} 1 & \mbox{if }&
\alpha\le\beta ,\\ 1-\alpha & \mbox{if }& \alpha>\beta, \end{array}\right.$\\
Gaines--Rescher   &
$I_{gr}(\alpha,\beta)=\left\{\begin{array}{lll}
1 & \mbox{if }& \alpha\le\beta ,\\ 0 & \mbox{if }&\alpha>\beta, \end{array}\right.$\\
Kleene--Dienes  & $I_b(\alpha,\beta)=\max\{1-\alpha,\beta\}$.\\
\hline
\end{tabular}

\bigskip The ''quality'' of these implication operators could be
evaluated either by empirically or by axiomatically methods.

Below we consider the implication operator defined in the
Lukasiewicz system of continuous-valued logic.

\begin{definition}\label{D5.1}\rm An interval valued  fuzzy set $F$ of $A$ is
called a {\it fuzzifying implicative filter} of $A$ if for any
$x,y,z\in A$ it satisfies the following four conditions:

\begin{enumerate}
\item[$(F_{18})$] $\models\big[ [x\in F]\wedge[y\in F] \rightarrow
[ x\odot y\in F]\,\big]$,
\item[$(F_{19})$] $\models\big[\,[x\in F]\rightarrow [y\in F]\,\big]$ for
any $x\le y$,
\item[$(F_{20})$] $\models\big[\,[(x\rightarrow y)\hookrightarrow x]
\rightarrow[x\in F]\,\big]$,
\item[$(F_{21})$] $\models\big[\,[(x\hookrightarrow y)\rightarrow x]
\rightarrow[x\in F]\,\big]$.
 \end{enumerate}\end{definition}

The concept of the ''standard'' tautology can be generalized to
the {\it $\widehat{\,t}$-tautology}, where $[0,0]<\widehat{\,t}\le
[1,1]$, in the following way:

\medskip
$\models_{\widehat{\,t}}  P\Longleftrightarrow [P]\ge
\widehat{\,t}$ \ for all valuations.

\medskip
This definition and results obtained in \cite{26} gives for us the
possibility to introduce such definition:

\begin{definition}\label{D5.2}\rm Let $[0,0]<\widehat{\,t}\le [1,1]$ be fixed. An interval valued  fuzzy set $F$ of
$A$ is called  a {\it $t$-implication-based interval valued fuzzy
implicative filter} of $A$ if for all $x,y,z\in A$ the following
conditions hold:

\begin{enumerate}
\item[$(F_{22})$] $\models_{\widehat{\,t}}\big[ [x\in F]\wedge[y\in
F] \rightarrow [x\odot y\in F]\,\big]$,
\item[$(F_{23})$] $\models_{\widehat{\,t}}\big[\,[x\in F]\rightarrow
[y\in F]\,\big]$ for all $x\le y$,
\item[$(F_{24})$] $\models_{\widehat{\,t}}\big[\,[(x\rightarrow
y)\hookrightarrow x] \rightarrow[x\in F]]$,
\item[$(F_{25})$] $\models_{\widehat{\,t}}\big[\,[(x\hookrightarrow
y)\rightarrow x] \rightarrow[x\in F]\,\big]$.
\end{enumerate}\end{definition}

In a special case when an implication operator is defined as $I$
we obtain:

\begin{corollary}\label{C5.3} An interval valued  fuzzy set $F$ of $A$
is a $t$-implication-based interval valued fuzzy implicative filter
if and only if for all $x,y,z\in A$ it satisfies:

\begin{enumerate}
\item[$(F_{26})$]
$I(\mathrm{rmin}\{\widehat{\mu_F}(x),\,\widehat{\mu_F}(y)\},\,\widehat{\mu_F}(x\odot
y))\ge \widehat{\,t}$,
\item[$(F_{27})$]
$I(\mathrm{rmin}\{\widehat{\mu_F}(x),\,\widehat{\mu_F}(y)\})\ge
\widehat{\,t}$ for all $x\le y$,
\item[$(F_{28})$] $I(\mathrm{rmin}\{\widehat{\mu_F}((x\rightarrow
y)\hookrightarrow x)\},\, \widehat{\mu_F}(x))\ge \widehat{\,t}$,
\item[$(F_{29})$] $I(\mathrm{rmin}\{\widehat{\mu_F}((x\hookrightarrow
y)\rightarrow x)\}, \,\widehat{\mu_F}(x))\ge \widehat{\,t}$.
\end{enumerate}
\end{corollary}

This gives a very good base for future study of filters in various
algebraic systems with implication operators. As an example we
present one theorem. In a similar way we can obtain other typical
results.

\begin{theorem}\label{T5.4} Let $F$ be an interval valued  fuzzy set of $A$.

\begin{enumerate}
\item[$(i)$] If $I=I_{gr}$, then $F$ is an
$ 0.5 $-implication-based interval valued fuzzy implicative filter
of $A$ if and only if $F$ is an interval valued  fuzzy implicative
filter with thresholds $(\widehat{r}=[0,0], \widehat{s}=[1,1])$.
\item[$(ii)$] If $I=I_g$, then $F$ is an $ 0.5 $-implication-based
fuzzy implicative filter of $A$ if and only if $F$ is an interval
valued  fuzzy implicative filter with thresholds
$(\widehat{r}=[0,0], \widehat{s}=[0.5,0.5])$.
\item[$(iii)$] If $I=I_{cg}$, then $F$ is an $ 0.5 $-implication-based
interval valued fuzzy implicative filter of $A$ if and only if $F$
is an interval valued  fuzzy implicative filter with thresholds
$(\widehat{r}=[0.5,0.5], \widehat{s}=[1,1])$.
\end{enumerate}
\end{theorem}
\begin{proof} The proofs are straightforward and hence
are omitted.
 \end{proof}

\section{Conclusions}
Interval valued fuzzy set theory emerges from the observation but
that in a number of cases, no objective procedure is available for
selecting the crisp membership degrees of elements in a fuzzy set.
It was suggested to alleviate that problem by allowing to specify
only an interval  to which the actual membership degree is assumed
to belong.  In this paper, we considered different type of interval
valued $(\in,\ivq)$-fuzzy filters of pseudo $BL$-algebras and
investigated the relationship between these filters. Finally, we
proposed the concept of implication-based interval valued fuzzy
implicative filters of pseudo $BL$-algebras, which seems to be a
good support for future study. The other direction of future study
is an investigation of interval valued $(\alpha,\beta)$-fuzzy
(implicative) filters, where $\alpha,\beta$ are one of $\in,q,\ivq$
or $\iwq.$

\subsection*{Acknowledgements }

The research was partially supported by a grant of the National
Natural Science Foundation of China , grant No.60474022 and a grant
of the Key Science Foundation of Education Committee of Hubei
Province, China, No. D200729003.

\end{document}